\documentclass[12pt,a4paper]{amsart}
\usepackage[utf8]{inputenc}
\usepackage[english]{babel}
\usepackage{latexsym,amssymb}
\usepackage{xcolor}
\usepackage{tensor}

\usepackage{cite}
\usepackage[pdftex]{graphicx}
\usepackage{tikz}
\usepackage{amsmath,amsthm,amsfonts,amssymb,bbm}
\usepackage{graphicx,psfrag,subfigure,float,color}
\usepackage{cite}
\usepackage{graphicx,hyperref}
\usetikzlibrary{arrows}
\usepgflibrary{snakes}
\usetikzlibrary{arrows,backgrounds}
\DeclareGraphicsExtensions{.jpg,.pdf,.pdftex,.eps}

\newcommand{\Pb}{\mathbb{P}}

\newcommand{\dx}{\mathrm{d}}
\newcommand{\R}{\mathbb{R}}

\newcommand{\Z}{\mathbb{Z}}

\newcommand{\GUE}{\mathrm{GUE}}

\newtheorem{tthm}{Theorem}

\newtheorem{prop}{Proposition}[section]

\newtheorem{cor}{Corollary}

\newtheorem{rem}[prop]{Remark}
\newenvironment{remark}{\begin{rem}\normalfont}{\end{rem}}

\theoremstyle{definition}

\newcommand{\blocktheorem}[1]{%
  \csletcs{old#1}{#1}
  \csletcs{endold#1}{end#1}
  \RenewDocumentEnvironment{#1}{o}
    {\par\addvspace{1.5ex}
     \noindent\begin{minipage}{\textwidth}
     \IfNoValueTF{##1}
       {\csuse{old#1}}
       {\csuse{old#1}[##1]}}
    {\csuse{endold#1}
     \end{minipage}
     \par\addvspace{1.5ex}}
}

\usepackage{graphicx}
\usepackage{amssymb}
\usepackage{epstopdf}
\usepackage{nicefrac}
\usepackage[margin=1.3in]{geometry}

\title{Cutoff profile of ASEP on a segment}

\author{Alexey Bufetov}

\address[Alexey Bufetov]{University of Leipzig, Germany. E-mail: alexey.bufetov@gmail.com}

\author{Peter Nejjar}

\address[Peter Nejjar]{Hausdorff Center for Mathematics \& Institute for Applied Mathematics, University of Bonn, Germany. E-mail: nejjar@iam.uni-bonn.de}

\begin{document}

\date{}

\begin{abstract}

This paper studies the  mixing behavior of the Asymmetric Simple Exclusion Process (ASEP) on a segment of length $N$. Our main result is that for particle densities in $(0,1),$ the total-variation cutoff window of ASEP is $N^{1/3}$ and the cutoff profile is $1-F_{\GUE},$ where $F_{\GUE}$ is the  Tracy-Widom distribution function.   This 
also gives a new proof of the cutoff itself, shown earlier
by Labb\'{e} and Lacoin. Our proof combines coupling arguments, the result of Tracy-Widom about fluctuations of ASEP started from the step initial condition, and exact algebraic identities coming from interpreting  the multi-species ASEP as a random walk on a Hecke algebra.
\end{abstract}

\maketitle

\section{Introduction}

We consider ASEP on the segment $[1;N]:=\{1,\ldots,N\}$ with $k\leq N$ particles. This is a continuous time Markov chain with  state space
\begin{equation*}
\Omega^{N,k} :=\left\{\xi\in \{0,1\}^{N}:\sum_{i=1}^{N}\xi(i)=k\right\}.
\end{equation*}
We think of the $1's$ as particles, and of the $0's$ as holes. The dynamics of ASEP can be described as follows: Each particle waits an exponential time with parameter $1$, after which with probability $p>1/2$ it attempts to make a unit step to the right, and with probability $q=1-p<1/2$ it attempts to make a unit step to the left. The attempt is succesfull if the target site lies in $[1;N]$ and is occupied by a hole, the hole and the particle exchanging their positions when the particle moves a unit step.  If the attempt is not successful, nothing happens.

For $\xi\in  \Omega^{N,k},$ we denote by $\xi_{t}$ the state at time $t$ of the ASEP started from $\xi,$ and we denote by $P_{t}^{\xi}$ the law of  $\xi_{t}$.
The ASEP dynamics on $[1;N]$ with $k$ particles has a unique stationary measure which we denote by $\pi_{N,k}$. 

Recall that the total-variation distance of two probability measures $\mu,\mu^{\prime}$ on $\Omega^{N,k}$ is given by
\begin{equation*}
||\mu-\mu^{\prime}||_{\mathrm{TV}} :=\max_{A \subset \Omega^{N,k}}|\mu(A)-\mu^{\prime}(A)|.
\end{equation*}
We define the maximal total-variation distance between the distribution at given time and the stationary distribution as
\begin{equation*}
d^{N,k}(t) :=\max_{\xi \in \Omega^{N,k}}||P_{t}^{\xi}-\pi_{N,k}||_{\mathrm{TV}},
\end{equation*}
and for $c\in \R, $ we define the time point
\begin{equation}
g(k,c):=\frac{(\sqrt{k}+\sqrt{N-k})^{2}+cN^{1/3}}{p-q}.
\end{equation}

The main result of this paper is the following:
\begin{tthm}\label{main} Assume that $k=k_{N}$ satisfies $\lim_{N\to\infty}k_{N}/N = \alpha$, and $\alpha \in (0,1)$. For any $c\in \R$ we have
\begin{equation*}
\lim_{N\to \infty}d^{N, k_{N}}\left(g(k_{N},c)\right) =1-F_{\GUE}(cf(\alpha)),
\end{equation*}
where $f(\alpha)=\frac{(\alpha(1-\alpha))^{1/6}}{(\sqrt{\alpha}+\sqrt{1-\alpha})^{4/3}}$, and $F_\GUE$ is the $\GUE$ Tracy-Widom distribution  defined in \eqref{FGUE2}.
\end{tthm}
\begin{proof}
This is an immediate consequence of Theorem \ref{upper} (which gives an upper bound for the limit on the lefthand side), proven in Section \ref{uppersec}, and Theorem \ref{lower} (which gives a lower bound), proven in Section \ref{lowersec}.
\end{proof}

Theorem \ref{main} gives the cutoff window and the cutoff profile (or shape) of ASEP, we refer to Chapter 18 of  the textbook  \cite{LPW17} by Levin-Peres (with contributions by Wilmer) for definitions and examples in the general context of Markov chains. A fortiori, Theorem \ref{main} also gives an independent proof of the cutoff itself, which was previously shown by Labb\'{e}-Lacoin in \cite[Theorem 2]{LL19}. On \cite[page 1556]{LL19} the authors mention that the cutoff window for the process is expected to be $N^{1/3}$ and that the cutoff profile is expected to be a function of the $\mathrm{Airy}_{2}$  process. Our Theorem \ref{main} confirms (and gives a precise meaning to) this conjecture. 

\subsection{Historic overview}

Detailed information about the relaxation of ergodic Markov chains to  equilibrium has been the goal of a vast literature, see e.g. classical works by Aldous \cite{A81}, Diaconis-Shahshahani\cite{DS81}, the review article by Diaconis \cite{D95}, the aforementioned textbook \cite{LPW17}, and references therein. Of particular interest is the so-called \textit{cutoff} phenomenon; a sequence of Markov chains exhibits this phenomenon if the distance between its distribution and the stationary measure abruptly falls from 1 to 0 on a certain time scale. There are different metrics for this distance, leading to different notions of cutoff, see the paper  \cite{HLP16} by Hermon-Lacoin-Peres  for their differences and similarities.  The most commonly used are separation and total-variation cutoff, we study the latter in this paper. Once the cutoff phenomenon is established, it is natural to ask for a more refined information: What happens at the critical time point on a finer scale? The answer to this question is given by the \textit{cutoff window}, which is a finer time scale at which the total-variation distance goes from 1 to 0 \textit{not} abruptly, and a \textit{cutoff profile}, which gives the exact limiting function for the total-variation distance in such a critical scaling. We refer to the article  \cite{Tes20} by Teyssier, and the work of Nestoridi-Thomas \cite{NT20} for recent interesting results about cutoff profiles. 

Cutoff-type questions were previously investigated for ASEP  as well. The first result in this direction was obtained by Diaconis-Ram in \cite{DR04}, where the  so called \textit{pre-cutoff} for a discrete time variation of ASEP was proved  with the use of representations of Hecke algebra. The pre-cutoff is a claim that there is a unique time scale in which the total-variation changes from 1 to 0, but not necessarily abruptly (again, we refer to \cite{LPW17} for formal definitions). In the work
\cite{BBHM} by Benjamini-Berger-Hoffman-Mossel, the pre-cutoff was shown for the ASEP on a segment, the ansatz of \cite{BBHM}  to  bound $d^{N,k}$  by studying hitting times is also used in the present work. Whether  cutoff holds was an open question for over a decade, until it was proven in \cite{LL19} with the use of hydrodynamics of ASEP on $\Z$ and a careful probabilistic analysis of the system.

It is important to mention that the $p=1,q=0$ case of Theorem \ref{main} can be obtained in a quite simple way from the result of  Johansson \cite[Theorem 1.6]{Jo00b}, see \cite[Theorem 1.6]{AHR09} by Angel-Holroyd-Romik. However, it was not clear how one can generalise such type of results  to $q \ne 0$. Our proof of Theorem \ref{main} can be viewed as such a generalisation.

Let us also mention that Theorem \ref{main} seems to be the first example when the fluctuation term $N^{1/3},$ which governs the Kardar-Parisi-Zhang universality class (see Corwin's review \cite{Cor11}), appears as a cutoff window in the study of mixing times, and the Tracy-Widom distribution appears as a cutoff profile.  


\subsection{Our tools and further questions}

Our proof combines several ingredients. The first one is the result \cite[Theorem 3]{TW08b} by Tracy-Widom about fluctuations of ASEP started from the step initial condition, see Theorem \ref{ASEPthm} below. Our argument is based on a comparison of the ASEP dynamics on a segment with the ASEP dynamics on all integers, and \cite[Theorem 3]{TW08b} is (not surprisingly) the original source of the function $F_\GUE$ in Theorem \ref{main}. 

The second ingredient is the use of the multi-species ASEP and its close connection to random walks on Hecke algebra. Using it, we are able to use certain symmetries of Hecke algebra in order to relate the ASEP started from the step initial condition and ASEP started from initial conditions that we are interested in. Somewhat similar ideas were used for TASEP ($q=0$ case) by Borodin-Bufetov in \cite{BB19} and Bufetov-Ferrari in  \cite{BF20} for the study of shocks. An important novelty of this paper is the extension of these ideas to ASEP case, which requires the use of \textit{Mallows elements} in Hecke algebra, see Section \ref{sec6} below. 

The third ingredient is a variety of probabilistic coupling techniques that are needed throughout the paper for all steps of the argument.  

As already mentioned earlier, our approach provides an independent (and rather short) proof of the cutoff for ASEP. Let us mention some further questions where our approach can be of use. 

ASEP is arguably the most well-known representative of a fairly large class of integrable stochastic systems in the KPZ universality class. One can study other systems from this class instead. For example, one can study mixing times of the so called q-TASEP on the interval $[1;N]$. It was shown by Bufetov in \cite{Bu20} that a variety of integrable systems can be interpreted as random walks on Hecke algebras. It is possible that our technique can be used to study the mixing times for them as well.

In a different direction, it is important to note that the aforementioned papers \cite{DR04}, \cite{BBHM}, \cite{LL19}, studied the mixing times of the multi-species ASEP (this process can be also referred to as a random Metropolis scan or a biased card shuffling) on a segment as well. As shown in \cite[Section 4.4]{LL19}, on the level of cutoff this question can be  reduced to a question about the single-species ASEP. The situation is significantly more delicate for the cutoff profile of the multi-species ASEP. Even for TASEP this question was resolved only very recently by Bufetov-Gorin-Romik in \cite{BGR20}; the obtained cutoff profile is the GOE Tracy-Widom distribution function (this is \textit{not} the $F_\GUE$ function from Theorem \ref{main}). Based on this, we conjecture that the cutoff profile for the multi-species ASEP is the GOE Tracy-Widom distribution function as well (and the cutoff window is $N^{1/3}$). Note that the result of \cite{BGR20} was based on highly nontrivial recent developments (see Borodin-Gorin-Wheeler \cite{BGW19}, Bisi-Cunden-Gibbons-Romik \cite{BCGR20}, Galashin \cite{G20}, Dauvergne \cite{D20} and Bufetov-Korotkikh \cite{BK20}). Parts of the approach from the current paper and from  \cite{BGR20} definitely can be of use for this question; nevertheless, this remains an interesting open problem. 

We also mention in Remark \ref{rem:further-questions} below further possible directions. 

\subsection*{Acknowledgments} We are grateful to anonymous referees for their helpful comments. The work of both authors was partially supported by the Deutsche Forschungsgemeinschaft (DFG, German Research Foundation) under Germany's Excellence Strategy -- EXC 2047 ``Hausdorff Center for Mathematics''. P. Nejjar is supported by the DFG  by the CRC 1060
(Projektnummer 211504053). Data sharing not applicable to this article as no datasets were generated or analysed during the current study.

\section{Preliminaries}





\subsection{Basic Coupling and graphical construction}\label{coupling}
While our main result deals with an ASEP on $[1;N]$ which has only particles and holes, it will be important for us to consider
ASEP on $\Z$ with countably many colors  of particles; instead of colors we may speak interchangeably of types, classes, or species of particles.
The state space of this multi-species ASEP is  the set of all bijections $\Z\to \Z$ which we denote by $\mathfrak{S},$ for $w\in \mathfrak{S},$ having $w(i)=j$ is interpreted as position $i$ being occupied by the particle with color $j$.
Our convention is that particles with lower color have priority over particles with higher color.

Here we briefly give  the graphical construction of this multi-species ASEP which  goes back to Harris \cite{Har78}. Let $(\mathcal{P}(z),z\in\Z)$ be a collection of independent, rate $p$ Poisson processes  constructed on some probability space $(\hat{\Omega},\mathcal{A}, \mathbb{P} ).$ By $\mathcal{P}_{t}(z)$ we denote the value of the Poisson process at time $t$.
For fixed $t$, the independence of the Poisson processes implies that  for almost every $\omega\in \hat{\Omega}$ there is a  sequence $(i_{n},n\in \Z)$ of integers such that
\begin{equation*}
\cdots <i_{n-2}< i_{n-1}<i_{n}<i_{n+1}<i_{n+2}<\cdots,
\end{equation*}
and $\mathcal{P}_{t}(i_{n})=0,n\in \Z.$

Given $w \in \mathfrak{S},$ and $z\in \Z$, we denote  the swapped configuration

\begin{equation*}
\sigma_{z,z+1}(w)(i)=
\begin{cases}
w(i+1)  &\mathrm{for} \, i=z\\
w(i-1)  &\mathrm{for} \, i=z+1\\
w(i)   &\mathrm{else}.
\end{cases}
\end{equation*}
The dynamics of ASEP is now as follows: We fix a (possibly random) initial configuration $w_{0}\in \mathfrak{S}$ and define the parameter
\begin{equation*}
Q:=\frac{q}{p}\in [0,1).
\end{equation*}  When at time $\tau$ the Poisson processes $\mathcal{P}(z)$ has a jump, we update the process as follows: If $w_{\tau^{-}}(z)<w_{\tau^{-}}(z+1)$, then we update $w_{\tau}=\sigma_{z,z+1}(w_{\tau^{-}})$, whereas if
 $w_{\tau^{-}}(z)>w_{\tau^{-}}(z+1), $ we toss an independent coin such that with probability $1-Q$, $w_{\tau}=w_{\tau^{-}},$ and with probability $Q$ we have  $w_{\tau}=\sigma_{z,z+1}(w_{\tau^{-}})$. Note that to construct the process up to time $t,$ it suffices to apply these update rules inside each of the finite boxes $[i_{n}+1;i_{n+1}],$ and inside each box there are a.s. only finitely many jumps of the Poisson processes (in particular, there is a.s. a well-defined first jump) during the interval $[0,t]$, and no two jumps happen at the same time, hence the graphical construction is well-defined. This construction also allows to obtain  the multi-species ASEP on a finite segment $[a;b]$ using finitely many Poisson processes $\mathcal{P}(a),\ldots, \mathcal{P}(b-1).$  This is a process on the set of permutations of $[a;b]$; we denote this set of permutations by $S_{a;b}$.

 To recover the ASEP which has only particles and holes, it suffices to fix an integer $k$ and identify all particles whose color lies in  $(-\infty,k]$ as particles, and all particles with color in $(k, +\infty)$ as holes, i.e. we map $w(\cdot)\mapsto 1_{(-\infty,k]}(w(\cdot)).$ This defines a map $X^{k}: \mathfrak{S}\to \{0,1\}^{\Z}$ (resp. a map $X_{[a;b]}^{k}: S_{a;b}\to \{0,1\}^{[a;b]}$ for the  finite ASEP), the image of which is the ASEP on $\Z$ (resp. $[a;b]$) with particles and holes.   In particular for  $k\in [1;N]$ we recover the ASEP in $\Omega^{N,k}$  described in the introduction.

  Later, we will also consider ASEPs with second class particles; they can be obtained by fixing $k_1, k_2\in \Z,k_{1}<k_{2},$ and identifying all particles whose color lies in  $(-\infty,k_1]$ as first class particles, all particles whose color lies in  $(k_1,k_2]$ as second  class particles, and all particles with color in $(k_2,\infty)$ as holes.

 The graphical construction  allows us to couple different ASEPs together: We use the same collection of Poisson processes to construct ASEPs which start from different initial configurations and/or at different time points. We call this coupling the \textit{basic coupling}, which also allows us to couple ASEPs on a segment  $[a;b]$ with ASEPs on $\Z$.
 \subsection{Invariant measure for ASEP }\label{secinv}

Define the \textit{Mallows measure} on $S_{a;b}$ as
$$
\mathcal{Q}_{a;b}(w) :=  Q^{(b-a+1)(b-a)/2 - l \left( w \right)} Z_{a;b},
$$
where $l(w)$ is the number of inversions of $w$ and
\begin{equation}\label{Zab}
Z_{a;b} := \left( \sum_{w \in S_{a;b}} Q^{(b-a+1)(b-a)/2 - l \left( w \right)} \right)^{-1} = \left( \sum_{w \in S_{a;b}} Q^{ l \left( w \right)} \right)^{-1} = \prod_{i=1}^{b-a+1} \frac{1-Q}{1-Q^i}.
\end{equation}
It is immediate that the Mallows measure is invariant for the multi-species ASEP on $[a;b]$ as described in Section \ref{coupling}. Furthermore,  under the map $X_{[a;b]}^{k}: S_{a;b}\to \{0,1\}^{[a;b]}$ (with $k\in [a;b]$) from  Section \ref{coupling}, the Mallows measure becomes the stationary measure for the ASEP on $[a;b]$ with $k-a+1$ particles and  $b-k$ holes which we denote by $\pi_{[a;b],k}$. For the case $[a;b] = [1;N]$ we will use a shortened notation $\pi_{[1;N],k}=:\pi_{N,k}.$

\subsection{Single-species ASEPs}
Given $\xi\in \Omega^{N,k}$, we attach a label (an integer) to each particle from right to left:
Let
\begin{equation*}
x_{k}^{\xi}(0)<\cdots < x_{1}^{\xi}(0)
\end{equation*}
be the initial positions of the $k$ particles of $\xi$. We denote by $x_{i}^{\xi}(t)$ the position at time $t$  of the particle that started in $x_{i}^{\xi}(0)$.

We will later often compare finite ASEPs with positive recurrent ASEPs on $\Z.$ The latter  are supported  on  $\bigcup_{Z\in \Z}\Omega_{Z},$ where for $Z\in \Z$
\begin{equation}\label{omega}
\Omega_{Z}=\left\{\zeta\in \{0,1\}^{\Z}:\sum_{j<Z}\zeta(j)=\sum_{j\geq Z}1-\zeta(j)<\infty\right\}.
\end{equation}
We note that
there is a partial order on $\Omega_{Z}$: For $\zeta^{\prime}, \zeta^{\prime\prime} \in \Omega_{Z}$, we define
\begin{equation}\label{order}
\zeta^{\prime}\preceq\zeta^{\prime\prime}\iff \sum_{j=r}^{\infty}1-\zeta^{\prime\prime }(j)\leq \sum_{j=r}^{\infty}1-\zeta^{\prime}(j)\quad \mathrm{for \, all\,}r\in \Z.
\end{equation}
It is easy to see that under the basic coupling, this order is preserved, i.e.  if $\zeta^{\prime}\preceq\zeta^{\prime\prime},$ then also $\zeta^{\prime}_{t}\preceq\zeta^{\prime\prime}_{t},t\geq 0.$

Likewise, we use the same symbol to denote the analogous partial order on $\{0,1\}^{[a;b]}$:
For $\xi^{\prime}, \xi^{\prime\prime}\in \{0,1\}^{[a;b]}$, we define
\begin{equation}\label{order2}
\xi^{\prime}\preceq\xi^{\prime\prime}\iff \sum_{j=r}^{b}1-\xi^{\prime\prime }(j)\leq \sum_{j=r}^{b}1-\xi^{\prime}(j)\quad \mathrm{for \, all\,}r\in [a;b].
\end{equation}

We denote the minimal and maximal element in $ \Omega^{N,k_{N}}$ w.r.t.  this order
as  $\xi^{0},\xi^1$:  \begin{equation}\label{xi01}
\xi^{0}=\mathbf{1}_{[1; k_{N}]}\in \Omega^{N,k_{N}},\quad \xi^{1}=\mathbf{1}_{[N-k_{N}+1; N]}\in \Omega^{N,k_{N}}.
\end{equation}
We define the corresponding elements in $\Omega_{N+1-k_{N}}$ as
\begin{equation}
\label{zeta01}\zeta^{0}=\mathbf{1}_{[1;k_{N}]}+\mathbf{1}_{\Z_{> N}}\in \Omega_{N+1-k_{N}},\quad \zeta^{1}=\mathbf{1}_{\Z_{>(N-k_{N})}}\in \Omega_{N+1-k_{N}}.
\end{equation}
Analogous to the finite case, we  label the particles of $\zeta^{0}$ from right to left.  We set
$\{i\in \Z: \zeta^{0}(i)=1\}=\{x_{i}^{\zeta^{0}}(0),i\leq k_{N}\}$ with $x_{k_{N}}^{\zeta^{0}}(0)< x_{k_{N}-1}^{\zeta^{0}}(0)<\cdots$  and denote by $x_{i}^{\zeta^{0}}(t)$ the position of the particle of $\zeta^{0}$ that started in $x_{i}^{\zeta^{0}}(0)$.
We also label the holes of $\zeta^{0}, $ but   from right to left, i.e. we write
$\{i\in \Z: \zeta^{0}(i)=0\}=\{H_{i}^{\zeta^{0}},i\leq k_{N}\}$ with $H_{k_{N}}(0)>H_{k_{N}-1}(0)>\cdots$, and $H_{i}^{\zeta^{0}}(t)$ is the  position at time $t$ of the hole that started in $H_{i}^{\zeta^{0}}(0).$ We label the particles and holes of $\zeta^{1}$ in the same way.

Given a particle configuration $\zeta\in \{0,1\}^{\Z}$
we define the (possibly infinite) position  of the leftmost particle and rightmost hole of
$\zeta$  as
\begin{equation}
\begin{aligned}
\label{parthole0}
\mathcal{L}(\zeta)=\inf\{i\in \Z:\zeta(i)=1\}\quad \mathcal{R}(\zeta)=\sup\{i\in \Z:\zeta(i)=0\}.
\end{aligned}
\end{equation}
We define analogously $\mathcal{L}(\xi),\mathcal{R}(\xi)$ for  $\xi \in \{0,1\}^{[a;b]}$.
With this notation,  we have in particular
\begin{equation}
\begin{aligned}
\label{parthole}
x_{k_{N}}^{\zeta^{0}}(t)=\mathcal{L}(\zeta_{t}^{0})\quad H_{k_{N}}^{\zeta^{0}}(t)=\mathcal{R}(\zeta_{t}^{0}).\end{aligned}
\end{equation}

Finally, throughout the paper, we will often omit writing integer brackets.


 \subsection{ASEP with step initial data}\label{stepsec}

Let us introduce the Tracy-Widom $\GUE$ distribution. This probability distribution originates in the theory of random matrices \cite{TW94}, namely  it is the limit law of the rescaled largest eigenvalue of
a matrix drawn from the Gaussian Unitary Ensemble (GUE). Its
cumulative distribution function is given by
\begin{equation}\label{FGUE2}
F_{\GUE}(s)=\sum_{n=0}^{\infty} \frac{(-1)^{n}}{n!}  \int_{s}^{\infty}\dx x_{1}\ldots  \int_{s}^{\infty}\dx x_{n}\det(K_{2}(x_{i},x_{j})_{1\leq i,j\leq n}), \end{equation}
where $K_{2}(x,y)$ is the Airy kernel  $K_{2}(x,y)=\frac{Ai(x)Ai^{\prime}(y)-Ai(y)Ai^{\prime}(x)}{x-y},x\neq y, $ defined for $x=y$ by continuity and $Ai$ is the Airy function.

The ASEP with step initial data is the ASEP on $\Z$ which starts from the initial configuration  $x_{m}^{\mathrm{step}}(0)=-m,m> 0$. The following fluctuation result plays an important role in our proof.

  \begin{tthm}[  \cite{TW08b}, Theorem 3 ] \label{ASEPthm}
Consider ASEP with step initial data and $p>q$. Let $\gamma=p-q,m>0, \sigma=m/t, c_1=1-2\sqrt{\sigma},c_2 =\sigma^{-1/6}(1-\sqrt{\sigma})^{2/3}$.
Then, uniformly for $\sigma $ in a compact subset of $(0,1),$ we have
\begin{equation}
\lim_{t \to \infty}\Pb\left(\frac{x_{m}^{\mathrm{step}}(t/\gamma)-c_1 t}{-c_{2}t^{1/3}}\leq s\right)=F_{\mathrm{GUE}}(s).
\end{equation}
\end{tthm}
We will need Theorem \ref{ASEPthm} in the form of the following corollary. Note that
the $c'N^{\kappa},c'' N^{\kappa'}$ terms in the following are irrelevant in the $N\to\infty$ limit as they get absorbed by the $N^{1/3}$ fluctuations.
 \begin{cor}\label{cor}
We have for $k_{N}$ with $k_N /N \to \alpha \in (0,1)$ and arbitrary $\kappa,\kappa'\in [0,1/3)$  and $c',c''\in \R$ that
\begin{equation}
\lim_{N \to \infty}\Pb\left(x_{k_{N}+c'N^{\kappa}}^{\mathrm{step}}(g(k_{N},c))\leq N-2k_{N} +c'' N^{\kappa'}\right)=1-F_{\mathrm{GUE}}(cf(\alpha)),
\end{equation}
where $f(\alpha)=\frac{(\alpha(1-\alpha))^{1/6}}{(\sqrt{\alpha}+\sqrt{1-\alpha})^{4/3}}.$
\end{cor}
\begin{proof}
We set $a=a(N)=k_{N}/N \to \alpha$ and $D=(\sqrt{a}+\sqrt{1-a})^{2}.$ Next we define $\tilde{N}=DN+cN^{1/3}=g(k_{N},c)(p-q) $. Writing everything in terms of $\tilde{N},$ we get (with $\sigma$ as in Theorem \ref{ASEPthm}) that \begin{equation*}\begin{aligned}
&k_{N}+c'N^{\kappa}=\frac{a}{D}\tilde{N}-\frac{ac}{D^{4/3}}\tilde{N}^{1/3}+o(\tilde{N}^{1/3})
\\&\sigma=  (k_{N}+c'N^{\kappa})/\tilde{N}=\frac{a}{D}-\frac{ac}{D^{4/3}}\tilde{N}^{-2/3}+o(\tilde{N}^{-2/3})
\\&N-2k_{N}+c''N^{\kappa'}=(1-2\sqrt{\sigma})\tilde{N}-c\tilde{N}^{1/3}\left(\frac{(1-2a)}{D^{4/3}}+\frac{\sqrt{a}}{D^{5/6}}\right)+o(\tilde{N}^{1/3}).
\end{aligned}\end{equation*}
We may thus rewrite
\begin{equation*}
\begin{aligned}
&\Pb\left(x_{k_{N}+c'N^{\kappa}}^{\mathrm{step}}(g(k_{N},c))\leq N-2k_{N} +c'' N^{\kappa'}\right)
\\&=1-\Pb\left(x_{k_{N}+c'N^{\kappa}}^{\mathrm{step}}(g(k_{N},c))>N-2k_{N} +c'' N^{\kappa'}\right)
\\&=1-\Pb\left(\frac{x_{\frac{a}{D}\tilde{N}-\frac{ac}{D^{4/3}}\tilde{N}^{1/3}+o(\tilde{N}^{1/3})}(\tilde{N}/(p-q))-(1-2\sqrt{\sigma})\tilde{N}}{-\tilde{N}^{1/3}}< c\left(\frac{(1-2a)}{D^{4/3}}+\frac{\sqrt{a}}{D^{5/6}}\right)+o(1)\right).
\end{aligned}
\end{equation*}
In order to apply Theorem \ref{ASEPthm} with $\tilde{N}=t,$ we still have to divide by $\sigma^{-1/6}(1-\sqrt{\sigma})^{2/3}$.
An elementary computation reveals 
\begin{equation*}
\frac{\left(\frac{(1-2a)}{D^{4/3}}+\frac{\sqrt{a}}{D^{5/6}}\right)}{\sigma^{-1/6}(1-\sqrt{\sigma})^{2/3}}=\frac{(a(1-a))^{1/6}}{D^{2/3}}+o(1)\to_{N\to\infty}\frac{(\alpha(1-\alpha))^{1/6}}{(\sqrt{\alpha}+\sqrt{1-\alpha})^{4/3}}= f(\alpha).
\end{equation*} 
Using Theorem \ref{ASEPthm} thus yields
\begin{align*}
&\lim_{\tilde{N}\to \infty}\Pb\left(\frac{x_{\frac{a}{D}\tilde{N}-\frac{ac}{D^{4/3}}\tilde{N}^{1/3}+o(\tilde{N}^{1/3})}(\tilde{N}/(p-q))-(1-2\sqrt{\sigma})\tilde{N}}{-\sigma^{-1/6}(1-\sqrt{\sigma})^{2/3}\tilde{N}^{1/3}}< c\frac{\left(\frac{(1-2a)}{D^{4/3}}+\frac{\sqrt{a}}{D^{5/6}}\right)}{\sigma^{-1/6}(1-\sqrt{\sigma})^{2/3}}+o(1)\right)
\\&=F_{\GUE}(cf(\alpha)),
\end{align*}
finishing the proof.
\end{proof}

\section{Upper  bound}\label{uppersec}
The aim of this section is to show that $1-F_{\GUE}(cf(\alpha))$ is an upper bound for the cutoff profile.    More precisely, we will show the following Theorem, see Figure \ref{Graph} for an illustration of the proof idea.
\begin{tthm}\label{upper}Let $k=k_{N}$ with $k_N /N \to \alpha \in (0,1)$. Then
we have for $c\in \R$
\begin{equation*}
\limsup_{N\to \infty}d^{N,k_{N}}\left(g(k_{N},c)\right) \leq 1-F_{\GUE}(cf(\alpha)).
\end{equation*}

\end{tthm}
The starting point for showing Theorem \ref{upper} is the following Theorem \ref{alexey}.
The proof of  Theorem \ref{alexey} exploits the link between multi-species ASEP and Hecke algebras, and is postponed to Section \ref{sec6}.

 \begin{tthm}\label{alexey}
Let  $\mathcal{L}(\zeta^{0}_{t}),\mathcal{R}(\zeta^{0}_{t})$  be as in \eqref{parthole}.
Define for $c\in \R$
\begin{equation*}
\begin{aligned}
B_{N}(c)=&\{   \mathcal{L}(\zeta^{0}_{g(k_{N},c)})>  N-k_N -N^{1/10}\}
\\&\cap    \{       \mathcal{R}(\zeta^{0}_{g(k_{N},c)})\leq N-k_N+N^{1/10}          \}.
\end{aligned}
\end{equation*}
Then
\begin{equation*}
\lim_{N\to\infty}\Pb(B_{N}(c))=F_{\GUE}(c f(\alpha)).
\end{equation*}
\end{tthm}
We note that the $N^{1/10}$ term   in the definition of $B_{N}(c)$  is there merely for concreteness, any term that goes to $+\infty$ with $N$ and is $o(N^{1/3})$ would do.

To get an upper bound for $d^{N,k_N  }(t),$ we consider the hitting time
\begin{equation}\label{HHH}
\mathfrak{H}=\inf\{t\geq 0: \zeta^{0}_{t}=\zeta^{1}\}.
\end{equation}
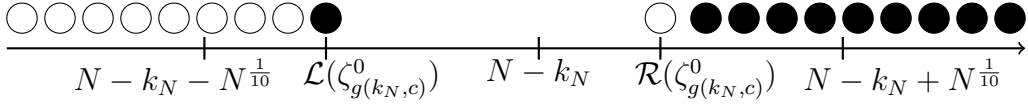
\begin{figure}\begin{center}
\begin{tikzpicture}[scale=2]

  \draw[thick, ->] (-1,0) -- (5.7,0);

    \draw (2.5,0) node[below] {$N-k_N$};
  \filldraw (1.1,0.2) circle (0.1);
      \draw (1.4,0) node[below] {$\mathcal{L}(\zeta^{0}_{g(k_{N},c)})$};

\draw (3.3,0.2) circle (0.1);
     \draw (3.6,0) node[below] {$\mathcal{R}(\zeta^{0}_{g(k_{N},c)})$};
   \draw (4.9,0) node[below] {$N-k_N+N^{\frac{1}{10}}$};
      \draw (0.1,0) node[below] {$N-k_N-N^{\frac{1}{10}}$};

       \foreach \x in {2.5,3.3,1.1,4.5,0.3}
       \draw[thick] (\x,0.075)--(\x,-0.075);
            \foreach \x in {3.6,3.85,4.1,4.35,4.6,4.85,5.1,5.35,5.6}
       \filldraw (\x,0.2) circle (0.1);
                 \foreach \x in {0.85,0.6,0.35,0.1,-0.15,-0.4,-0.65,-0.9}
       \draw (\x,0.2) circle (0.1);

\end{tikzpicture}\end{center}
\caption{The particle configuration $\zeta_{g(k_{N},c)}^{0}$ on the event $B_{N}(c)$ from Theorem \ref{alexey}: Black balls indicate the presence of a particle, white balls of a hole, blank space may be occupied by  holes or particles.  On $B_{N}(c),$ the leftmost particle $\mathcal{L}(\zeta^{0}_{g(k_{N},c)})$ and the rightmost hole $\mathcal{R}(\zeta^{0}_{g(k_{N},c)})$ of $\zeta^{0}_{g(k_{N},c)}$ lie in $[N-k_{N}-N^{1/10};N-k_{N}+N^{1/10}].$  If $B_{N}(c)$ happens, the hitting time $\mathfrak{H}$ from \eqref{HHH} cannot be much larger than $g(k_{N},c)$, see Proposition \ref{0}. As $\mathfrak{H}$
induces an upper bound for $d^{N,k_N}(g(k_{N},c))$ via Proposition \ref{hineq2}, this implies that $d^{N,k_N}(g(k_{N},c))$ is asymptotically bounded from above by $1-\Pb(B_{N}(c))$, which equals $1-F_{\GUE}(cf(\alpha))$ asymptotically.
 }
\label{Graph}
\end{figure}

The link between this hitting time and the maximal  total-variation distance is as follows.
\begin{prop}\label{hineq2}
We have
\begin{equation}\label{ineqcoupl}
d^{N,k_N  }(t)\leq \Pb(\mathfrak{H}> t).
\end{equation}
\end{prop}
\begin{proof}
Let $\xi, \xi^{\prime}\in \Omega^{N,k}$ and let the ASEPs $(\xi_{t},\xi^{\prime}_{t},t\geq 0)$ be coupled via the basic coupling.  We define the coalescence time
\begin{equation}
\tau^{\xi,\xi^{\prime}}=\inf\{t: \xi_{t}=\xi^{\prime}_{t}\}.
\end{equation}

Then we have the  general  inequality (see  \cite[Corollary 5.5]{LPW17})
\begin{equation}\label{5.5}
d^{N,k_{N}}(t)\leq \max_{\xi,\xi^{\prime}\in \Omega^{N,k_{N}}}\Pb(\tau^{\xi,\xi^{\prime}}>t).
\end{equation}
Define furthermore the hitting time
\begin{equation}
\mathfrak{h}=\inf\{t:\xi^{0}_{t}=\xi^{1}\}.
\end{equation}
Assuming all appearing ASEPs are coupled via the basic coupling, we prove the inequality
\begin{equation}\label{hineq}
\max_{\xi,\xi^{\prime}\in \Omega^{N,k_{N}}}\tau^{\xi,\xi^{\prime}}\leq \mathfrak{h}.
\end{equation}
For this,  recall the partial order \eqref{order2}.  We then have
\begin{equation}
\xi^{1}=\xi^{0}_{\mathfrak{h}}\preceq \xi_{\mathfrak{h}}\preceq \xi^{1}
\end{equation}
i.e. $\xi_{\mathfrak{h}}=\xi^{1}$ and likewise $\xi_{\mathfrak{h}}^{\prime}=\xi^{1}$ so that  $\xi_{\mathfrak{h}}=\xi_{\mathfrak{h}}^{\prime}$ and $\tau^{\xi,\xi^{\prime}}\leq \mathfrak{h}$.

To proceed, we use the basic  coupling of  ASEPs on $[1;N]$ with ASEPs on $\Z$ described at the end of Section \ref{coupling}. 

We  will show that
\begin{equation}\label{hH}
\mathfrak{h}\leq \mathfrak{H}.
\end{equation}
To see this, we show  the following random time is infinite: Let 
\begin{equation*}
\mathcal{T}=\inf\{t: \mathrm{\, there\, is\, an \,}i^{*}\in\{1,\ldots,k_{N}\}\mathrm{\,such \,that\,}  x^{\zeta^{0}}_{i^*}(t)>x^{\xi^{0}}_{i^*}(t)\}.
\end{equation*}
Note that $x^{\zeta^{0}}_{i}(0)=x^{\xi^{0}}_{i}(0),i=1,\ldots,k_{N}.$ Thus, to have an $i^{*}$  with $x^{\zeta^{0}}_{i^*}(t)>x^{\xi^{0}}_{i^*}(t),$ one of the Poisson processes $(\mathcal{P}(z),z\in[1;N])$ of the graphical construction must have made a jump during $[0,t].$ In particular, $\mathcal{T}>0$ almost surely, we can thus consider the left limit $\mathcal{T}^{-}.$  At time $\mathcal{T}$, there is exactly one $i^{*}$ with
$x^{\zeta^{0}}_{i^*}(\mathcal{T})>x^{\xi^{0}}_{i^*}(\mathcal{T}),$  having more than one  $i^{*}$ would require two jumps to happen at the same time. Furthermore, we always have
\begin{equation}
x^{\xi^{0}}_{i^*}(\mathcal{T}^{-})=x^{\zeta^{0}}_{i^*}(\mathcal{T}^{-}), \quad x^{\xi^{0}}_{i}(\mathcal{T}^{-})\geq x^{\zeta^{0}}_{i}(\mathcal{T}^{-}), i=1,\ldots,k_{N}
\end{equation}

We now distinguish two possibilities : The first possibility  to have $x^{\zeta^{0}}_{i^*}(\mathcal{T})>x^{\xi^{0}}_{i^*}(\mathcal{T})$ is that at time $\mathcal{T},$ $x^{\zeta^{0}}_{i^*}$ makes a jump to the right that $x^{\xi^{0}}_{i^*}$ does not make.  One way for $x^{\xi^{0}}_{i^*}$  not to make a jump to the right is if $x^{\xi^{0}}_{i^*}(\mathcal{T}^{-})=N.$  But then also $x^{\zeta^{0}}_{i^*}(\mathcal{T}^{-})=N,$ and $x^{\zeta^{0}}_{i^*}$ cannot jump either. The other way for
$x^{\xi^{0}}_{i^*}$  not to make a jump to the right is if it is blocked by the particle  $x^{\xi^{0}}_{i^* -1}$, i.e. if
\begin{equation*}
x^{\xi^{0}}_{i^* -1}(\mathcal{T}^{-})=x^{\xi^{0}}_{i^* }(\mathcal{T}^{-})+1.
\end{equation*}

 However, since \begin{equation*}x^{\xi^{0}}_{i^* -1}(\mathcal{T}^{-})\geq x^{\zeta^{0}}_{i^* -1}(\mathcal{T}^{-})> x^{\zeta^{0}}_{i^* }(\mathcal{T}^{-})=x^{\xi^{0}}_{i^* }(\mathcal{T}^{-}),\end{equation*}
 this implies that
 \begin{equation*}
x^{\zeta^{0}}_{i^* -1}(\mathcal{T}^{-})=x^{\zeta^{0}}_{i^* }(\mathcal{T}^{-})+1,
\end{equation*}

showing that $x^{\zeta^{0}}_{i^*}$  cannot jump to the right  at time $\mathcal{T}$ either.

The second possibility to have $x^{\zeta^{0}}_{i^*}(\mathcal{T})>x^{\xi^{0}}_{i^*}(\mathcal{T})$ is that at time $\mathcal{T},$ $x^{\xi^{0}}_{i^*}$ makes a jump to the left that $x^{\zeta^{0}}_{i^*}$ does not make. The only way for this to happen however is that $x^{\zeta^{0}}_{i^* +1}$  blocks the left jump of $x^{\zeta^{0}}_{i^*}$, i.e.

\begin{equation}
x^{\zeta^{0}}_{i^* +1}(\mathcal{T}^{-})=x^{\zeta^{0}}_{i^* }(\mathcal{T}^{-})-1.\end{equation}

But since $x^{\xi^{0}}_{i^* +1}(\mathcal{T}^{-})\geq x^{\zeta^{0}}_{i^* +1}(\mathcal{T}^{-}),$ this implies that
\begin{equation}x^{\xi^{0}}_{i^* +1}(\mathcal{T}^{-})=x^{\xi^{0}}_{i^* }(\mathcal{T}^{-})-1,
\end{equation}
and therefore $x^{\xi^{0}}_{i^*}$ cannot jump to the left at time $\mathcal{T}$ either.

In total, we thus have $\Pb(\mathcal{T}=\infty)=1$. Since clearly

\begin{equation}
x_{i}^{\zeta^{0}}(\mathfrak{H})=N+1-i,\quad i=1,\ldots,k_{N},
\end{equation}
we can thus conclude
\begin{equation}
x_{i}^{\xi^{0}}(\mathfrak{H})=N+1-i,\quad i=1,\ldots,k_{N},
\end{equation}
and therefore \eqref{hH} holds.  Combining \eqref{5.5}, \eqref{hineq} and \eqref{hH} finishes the proof.
\end{proof}

The key observation now is that if the event $B_{N}(c)$ from Theorem \ref{alexey} happens, the hitting time $\mathfrak{H}$ cannot be much larger than $g(k_{N},c)$.

In order to show this, we will use the following result from \cite{BBHM}, which in fact is much stronger than what we will be needing.
\begin{tthm}[Theorem 1.9 of \cite{BBHM}]\label{BBHMthm}For $M\in \Z_{\geq 1}$, let $I^M=\mathbf{1}_{[-M;-1]}+\mathbf{1}_{\Z_{\geq M}},$ and $(I^{M}_{t},t\geq 0)$ be the ASEP started from $I^{M}$.  Define the hitting time $\mathfrak{H}_{M}=\inf\{t\geq 0: I^{M}_{t}=\mathbf{1}_{\Z_{\geq 0}}\}$. Then for every $\delta>0$ there is a constant $D=D(p,\delta)$ such that 
\begin{equation}\label{26}
\Pb( \mathfrak{H}_{M} \geq DM)<\frac{\delta}{M}.
\end{equation}

\end{tthm}
Now we can show that on $B_{N}(c)$, we have good control over $\mathfrak{H}$.
\begin{prop}\label{0} We have
\begin{equation}
 \lim_{N\to\infty}\Pb(\{\mathfrak{H}\geq g(k_{N},c)+N^{1/5}\}\cap B_{N}(c))=0.
\end{equation}
\end{prop}

\begin{proof}We define with $\widetilde{N}:=N-k_N+N^{1/10}+1$ the configuration
\begin{equation}
\eta (j)= \mathbf{1}_{[N-k_N -N^{1/10}+1; N-k_N]}+\mathbf{1}_{\Z_{\geq \widetilde{N}}}.
\end{equation}

We start now \textit{at time} $g(k_{N},c)$ an ASEP from $\eta$ and couple this ASEP to all other appearing ASEPs via the basic coupling.
To make this clear in our notation, we write
\begin{equation}
\hat{\eta}_{g(k_{N},c)}:=\eta
\end{equation}
and denote $(\hat{\eta}_{\ell},\ell\geq g(k_{N},c))$ the ASEP which starts at time
$g(k_{N},c)$ from $\eta$, so that $\hat{\eta}_{t}$ for $t\geq g(k_{N},c)$ has the same law as $\eta_{t- g(k_{N},c)}.$ We define the corresponding hitting time
\begin{equation}
\mathfrak{H}^{\eta}=\inf\{t\geq g(k_{N},c): \hat{\eta}_{t}=\zeta^{1}\}.
\end{equation}
With $I^{N^{\frac{1}{10}}},\mathfrak{H}_{N^{\frac{1}{10}}}$ defined in Theorem \ref{BBHMthm}, we have $\eta(j+N-k_N+1)=I^{N^{\frac{1}{10}}}(j),j\in \Z,$ and thus in particular,
$\mathfrak{H}^{\eta}$ has the same law as $\mathfrak{H}_{N^{\frac{1}{10}}}+g(k_{N},c)$.
It is thus an immediate  corollary of \eqref{26} that for every $\varepsilon>0$ we have
\begin{equation}\label{22}
\lim_{N\to \infty}\Pb(\mathfrak{H}^{\eta}\geq g(k_{N},c)+N^{\frac{1}{10}+\varepsilon})=0.
\end{equation}

It is  easy to see  that we have the inclusion
\begin{equation}
B_{N}(c)\subseteq \{\zeta^{0}_{g(k_N,c)}\succeq\hat{\eta}_{g(k_{N},c)}\}.
\end{equation}
Since the partial order $\succeq$ is preserved under the basic coupling, we thus have
\begin{equation}
B_{N}(c)\subseteq \{\mathfrak{H}^{\eta}\geq \mathfrak{H}\}.
\end{equation}
Taking $\varepsilon >0$ such that $1/10+\varepsilon<1/5$, we can thus conclude from \eqref{22} that
\begin{equation}
\begin{aligned}
& \lim_{N\to\infty}\Pb(\{\mathfrak{H}\geq g(k_{N},c)+N^{1/5}\}\cap B_{N}(c))
\\&\leq  \lim_{N\to\infty}\Pb(\{\mathfrak{H}^{\eta}\geq g(k_{N},c)+N^{1/5}\})=0.
\end{aligned}
\end{equation}

\end{proof}

We can now prove Theorem \ref{upper}.

\begin{proof}[Proof of Theorem \ref{upper}]
Let $\varepsilon>0$ and let $N$ be sufficiently large so that \begin{equation*}g(k_N,c-\varepsilon)+N^{1/5}<g(k_{N},c).\end{equation*} Then, using \eqref{ineqcoupl}, we get
\begin{align*}
d^{N,k_N}\left(g(k_{N},c)\right) &\leq\Pb(\{\mathfrak{H}\geq g(k_N,c-\varepsilon)+N^{1/5}\}\cap B_{N}(c-\varepsilon))\\&+1-\Pb(B_{N}(c-\varepsilon)).\end{align*}
Combining Theorem \ref{alexey} with Proposition \ref{0} we get
\begin{align*}
\limsup_{N\to \infty}d^{N,k_N}\left(g(k_{N},c)\right) \leq \lim_{N\to \infty}1-\Pb(B_{N}(c-\varepsilon))=1-F_{\GUE}((c-\varepsilon)f(\alpha)).\end{align*}
Since $\varepsilon $ is arbitrary, Theorem \ref{upper} follows.
\end{proof}

\section{Lower bound}\label{lowersec}
In this section we prove  that the upper bound obtained in the previous section is also a lower bound. This is the content of the following Theorem.
\begin{tthm}\label{lower}Let $k=k_{N}$ with $k_N /N \to \alpha \in (0,1)$. Then
we have for $c\in \R$
\begin{equation}
\liminf_{N\to \infty}d^{N,k_{N}}\left(g(k_{N},c)\right) \geq 1-F_{\GUE}(cf(\alpha)),
\end{equation}
where $f(\alpha)=\frac{(\alpha(1-\alpha))^{1/6}}{(\sqrt{\alpha}+\sqrt{1-\alpha})^{4/3}}.$

\end{tthm}
The main tool to prove Theorem \ref{lower} is to compare the finite ASEP $(\xi_{t}^{0},t\geq 0)$ (defined in  \eqref{xi01}) with the infinite ASEP started from the step initial data from Section \ref{stepsec}.
We will  consider the shifted step initial data on $\Z$ given by
\begin{equation}\label{shift}
x_{n}^{\mathrm{step}(k_{N})}(0)=k_{N}+1-n, n\geq 1.
\end{equation}
We start by noting that under the basic coupling,  the leftmost particle of $\xi^{0}_{t}$ is never to the right of $ x_{k_{N}}^{\mathrm{step}(k_{N})}(t)$.
\begin{prop}\label{3.1}
Consider the basic coupling of the finite ASEP $(\xi^{0}_{t},t\geq 0)$ and the infinite
ASEP started from the shifted step initial data \eqref{shift}. Then we have
\begin{equation}
x_{k_{N}}^{\xi^{0}}(t)\leq x_{k_{N}}^{\mathrm{step}(k_{N})}(t), \quad t \geq 0.
\end{equation}

\end{prop}
\begin{proof}
This is quite similar to the proof of the inequality \eqref{hH}, we will thus not repeat all the details. We show $\Pb(\widehat{\mathcal{T}}=\infty)=1,$ where  $\widehat{\mathcal{T}}$ is the random time
\begin{equation*}
\widehat{\mathcal{T}}=\inf\{t: \mathrm{\, there\, is\, an \,}i^{*}\in\{1,\ldots,k_{N}\}\mathrm{\,such \,that\,}  x^{\mathrm{step}(k_{N})}_{i^*}(t)<x^{\xi^{0}}_{i^*}(t)\}.
\end{equation*}

We again  distinguish two possibilities : The first possibility  to have $x^{\mathrm{step}(k_{N})}_{i^*}(\widehat{\mathcal{T}})<x^{\xi^{0}}_{i^*}(\widehat{\mathcal{T}})$ is that at time $\widehat{\mathcal{T}},$ $x^{\xi^{0}}_{i^*}$ makes a jump to the right that $x^{\mathrm{step}(k_{N})}_{i^*}$ does not make.  This implies that
\begin{equation}
x^{\mathrm{step}(k_{N})}_{i^* -1}(\widehat{\mathcal{T}}^{-})=x^{\mathrm{step}(k_{N})}_{i^* }(\widehat{\mathcal{T}}^{-})+1,
\end{equation}
which however implies that
\begin{equation}
x^{\xi^{0}}_{i^* -1}(\widehat{\mathcal{T}}^{-})=x^{\xi^{0}}_{i^* }(\widehat{\mathcal{T}}^{-})+1
\end{equation}
also, meaning $x^{\xi^{0}}_{i^* }$ cannot jump to the right at time
$\widehat{\mathcal{T}}$ either.

The other possibility is that at time $\widehat{\mathcal{T}}$, $x^{\mathrm{step}(k_{N})}_{i^*}$ makes a jump to the left that $x^{\xi^{0}}_{i^*}$ does not make. For $x^{\xi^{0}}_{i^*}$ not to make a jump to the left, we must have either
$x^{\xi^{0}}_{i^*}(\widehat{\mathcal{T}}^{-})=1,$  or we have
\begin{equation}\label{40}
x^{\xi^{0}}_{i^* +1}(\widehat{\mathcal{T}}^{-})=x^{\xi^{0}}_{i^* }(\widehat{\mathcal{T}}^{-})-1.
\end{equation}

If $x^{\xi^{0}}_{i^*}(\widehat{\mathcal{T}}^{-})=1,$ we have $i^{*}=k_{N},$ and $x^{\mathrm{step}(k_{N})}_{k_{N}}(\widehat{\mathcal{T}}^{-})=1,$ so that $x^{\mathrm{step}(k_{N})}_{k_{N}}$ also  cannot jump to the left at time $\widehat{\mathcal{T}},$ since $x^{\mathrm{step}(k_{N})}_{k_{N}}$ can never jump from $1$ to $0.$
If instead \eqref{40} holds,  we also have
\begin{equation}\label{41}
x^{\mathrm{step}(k_{N})}_{i^* +1}(\widehat{\mathcal{T}}^{-})=x^{\mathrm{step}(k_{N})}_{i^* }(\widehat{\mathcal{T}}^{-})-1,
\end{equation}
implying that $x^{\mathrm{step}(k_{N})}_{i^* }$ cannot jump to the left at time $\widehat{\mathcal{T}}$ either.

\end{proof}
Define for $l\in [1;N-k-1]$ the event
\begin{equation}
A_{N}(l)=\left\{\xi\in \Omega^{N,k}:\sum_{i=N-k-l}^{N}\xi(i)\leq k-1\right\},
\end{equation}
so that
\begin{equation*}
\xi_{t} \in A_{N}(l) \iff \mathcal{L}(\xi_t) <N-k-l.
\end{equation*}
Recall that $\pi_{N,k}$ is the stationary measure of ASEP in $\Omega^{N,k}$. The next proposition shows in particular   that $\pi_{N,k}$ gives vanishing mass to the event $A_{N}(l) $ when $l=l(N)$ goes to $+\infty$ with $N$.

\begin{prop}\label{3.2}
There are  constants $C_1,C_2>0$ which depend on $p$ but not on $N,k_N$  such that we have  $\pi_{N,k_{N}}(A_{N}(l))\leq C_{1}e^{-C_{2}l}.$
\end{prop}
\begin{proof}
We consider the basic coupling and  define the random time
\begin{equation*}
\widetilde{\mathcal{T}}=\inf\{t: \mathrm{\, there\, is\, an \,}i^{*}\in\{1,\ldots,k_{N}\}\mathrm{\,such \,that\,}  x^{\xi^{1}}_{i^*}(t)< x^{\zeta^{1}}_{i^*}(t)\}.
\end{equation*}
With a  proof that is very similar to the proof of \eqref{hH} and Proposition \ref{3.1}, we show that  $\Pb(\widetilde{\mathcal{T}}=\infty)=1.$ 
Thus in particular  $\mathcal{L}(\xi^{1}_t)\geq \mathcal{L}(\zeta^{1}_t)$ holds for all $t\geq 0,$  and hence
\begin{align*}
\pi_{N,k_{N}}(A_{N}(l))&=\lim_{t\to\infty}\Pb(\mathcal{L}(\xi^{1}_t)<N-k_{N}-l)\\&\leq \lim_{t\to\infty}\Pb( \mathcal{L}(\zeta^{1}_t)<N-k_{N}-l).
\end{align*}
Finally, we apply \cite[Proposition 3.1]{N20CMP} which shows that
\begin{align*}
 \lim_{t\to\infty}\Pb( \mathcal{L}(\zeta^{1}_t)<N-k_{N}-l)\leq C_{1}e^{-C_{2}l},
\end{align*}
finishing the proof.
\end{proof}
Now we can prove Theorem \ref{lower}.
\begin{proof}[Proof of Theorem \ref{lower}]
Recall $P_{t}^{\xi}$ is the law of the ASEP started from $\xi$ at time $t$. We can by definition   bound
\begin{equation}
d^{N,k_{N}}(g(k_{N},c))\geq P_{g(k_{N},c)}^{\xi^{0}}(A_{N}(N^{1/4}))-\pi_{N,k_{N}}(A_N (N^{1/4})).
\end{equation}
By Proposition \ref{3.2}, $\lim_{N\to \infty}\pi_{N,k_{N}}(A_N (N^{1/4}))=0.$
Combining Proposition \ref{3.1} with Corollary \ref{cor} yields
\begin{equation}
\begin{aligned}
\liminf_{N\to\infty} P_{g(k_{N},c)}^{\xi^{0}}(A_{N}(N^{1/4}))&=\liminf_{N\to\infty}\Pb( x_{k_{N}}^{\xi^{0}}(g(k_{N},c))<N-k_{N}-N^{1/4})
\\&\geq\lim_{N\to\infty}\Pb( x_{k_{N}}^{\mathrm{step}(k_{N})}(g(k_{N},c))<N-k_{N}-N^{1/4})
\\&=1-F_{\mathrm{GUE}}(cf(\alpha)),
\end{aligned}
\end{equation}
finishing the proof.
\end{proof}

\section{Proof of Theorem \ref{alexey}}\label{sec6}
In this section we prove Theorem \ref{alexey} via a certain distribution identity coming from viewing the multi-species ASEP as a random walk on a Hecke algebra. We start by introducing the Hecke algebra and other necessary notions in Sections \ref{subsec:rand-walkHA} -- \ref{Qeq}. Then in Section \ref{plan} we explain the idea of the proof. In the remaining sections we give a formal proof. 



\subsection{Random walk on Hecke algebra}
\label{subsec:rand-walkHA}

Let $S_n$ be the symmetric group of permutations of $n$ elements. For each permutation $w \in S_n,$ recall that we denote by $l(w)$ the number of inversions in it. Let $\mathfrak{S}_n$ be the set of all nearest neighbor transpositions from $S_n$.

We will fix the parameter (introduced earlier in Section \ref{coupling})
\begin{equation*}
Q:=\frac{q}{p}\in [0,1).
\end{equation*}
A \textit{Hecke algebra} $\mathcal H (S_n)$ is the algebra with a linear basis $\{ T_w \}_{w \in S_n}$ and the multiplication which satisfies the following rules for any $s \in \mathfrak{S}_n$, $w \in S_n$:
\begin{equation}
\label{eq:HeckeRules}
\begin{cases}
T_s T_w = T_{sw}, \qquad & \mbox{if $l(sw)=l(w)+1$}  \\
T_s T_w = (1-Q) T_w + Q T_{sw}, \qquad & \mbox{if $l(sw)=l(w)-1$}.
\end{cases}
\end{equation}
It is clear that such rules can be used for a computation of the product $T_{w_1} T_{w_2}$ for any $w_1, w_2 \in S_n$; a non-trivial (but very well-known) part is that the rules are consistent and indeed define an (associative) multiplication.

Let $\mathfrak i: \mathcal H (S_n) \to \mathcal H (S_n)$ be a linear map such that $\mathfrak i \left( T_w \right) = T_{w^{-1}}$. The following proposition is well-known (and can be straightforwardly proved by induction in $l(w)$ with the use of \eqref{eq:HeckeRules}).

\begin{prop}
\label{prop:CPsymmetry}
The map $\mathfrak i$ is an involutive anti-homomorphism. In more detail, for any $T_1, \dots T_r \in \mathcal H (S_n)$ we have
\begin{equation*}
\mathfrak i \left( T_r T_{r-1} \dots T_2 T_1 \right) = \mathfrak i \left( T_1 \right) \mathfrak i \left( T_2 \right) \dots \mathfrak i \left( T_{r-1} \right) \mathfrak i \left( T_r \right),
\end{equation*}
and also, trivially, $\mathfrak i^2 \left( T_1 \right) = T_1$.
\end{prop}

For $a,b \in \Z$, $a<b$, recall that  $S_{a;b}$ is  the group of permutations of the set $[a;b]$. We have the natural embedding $S_{a_2;b_2} \subset S_{a_1;b_1}$ for any $a_1 \le a_2 \le b_2 \le b_1$. Denote by $\mathcal H_{a;b}$ the Hecke algebra corresponding to $S_{a;b}$. These Hecke algebras satisfy an analogous embedding relation.

Consider the following \textit{random walk on the Hecke algebra} $\mathcal H_{a;b}$. Let $\mathfrak{S}_{a;b}$ be the set of all nearest neighbor transpositions from $S_{a;b}$. We attach to every element $(z,z+1)$  of $\mathfrak{S}_{a;b}$ a Poisson process $\mathcal{P}(z)$ on $\R_{\ge 0}$ of rate $p$. All these Poisson processes are jointly independent. Next, we define the stochastic process $W_{a;b} (t)$ which takes values in $\mathcal H_{a;b}$. Its initial value is $W_{a;b} (0) = T_{id}$ (the basis vector corresponding to the identity permutation). When at a certain time $\tau \in \R_{\ge 0}$ one of the Poisson processes has a point, then we set $W_{a;b} (\tau) = T_s W_{a;b} (\tau^{-})$, where $s$ is the nearest neighbor transposition corresponding to the Poisson process  with point at $\tau$. Since the points from all Poisson processes are almost surely distinct and can be linearly ordered, this rule defines a   stochastic process $W_{a;b} (t)$ in continuous time.

We will need  the \textit{Mallows element}
$$
\mathcal{M}_{a;b} := \sum_{w \in S_{a;b}} Q^{(b-a+1)(b-a)/2 - l \left( w \right)} Z_{a;b} T_w, \qquad \mathcal{M}_{a;b} \in \mathcal H_{a;b},
$$
with $Z_{a;b}$ as in \eqref{Zab} and furthermore we define
 \begin{equation}
\mathcal{H}_{\mathrm{prob}}(S_{a;b})=\left\{h \in \mathcal H_{a;b}: h=\sum_{w\in  S_{a;b}} \kappa_w T_w,  \kappa_w \ge 0,  \sum_{w\in S_{a;b}} \kappa_w = 1\right\}.
\end{equation}
The main property of the element $\mathcal{M}_{a;b}$ is\footnote{It is sufficient to check this property for $h_{a;b}=T_s$, for any $s \in \mathfrak{S}_{a;b}$. For such a choice it follows from a detailed balance type equation.}
\begin{equation}\label{Melement}
h_{a;b} \mathcal{M}_{a;b} = \mathcal{M}_{a;b} h_{a;b} = \mathcal{M}_{a;b}, \qquad \mbox{for any $h_{a;b} \in \mathcal{H}_{\mathrm{prob}}(S_{a;b})$}.
\end{equation}

Finally note that due to the multiplication rule \eqref{eq:HeckeRules} the elements $\mathcal{M}_{a;b}$, $W_{a;b} (t)$, as well as their products, are elements of $\mathcal{H}_{\mathrm{prob}}(S_{a;b})$.
For any element of $\mathcal{H}_{\mathrm{prob}}(S_{a;b})$ one can define the random permutation \textit{generated} by this element of Hecke algebra by assigning to a permutation $w$ the probability $\kappa_w$.

\subsection{Multi-species ASEP as a random walk on Hecke algebra}
Let us make the link here between the random walk $W_{a;b}(t)$ and the multi-species ASEP on  $S_{a;b}$ as constructed in Section \ref{coupling}. Note that
the process $(W_{a;b}(t),t\geq0),$ which takes values in the Hecke algebra $\mathcal H_{a;b},$ immediately  induces  a stochastic process $(w_{t},t\geq0)$ on $S_{a;b}$.  The definition of $W_{a;b}(t)$ and the multiplication rules \eqref{eq:HeckeRules} imply that this process is exactly the multi-species ASEP as introduced in Section \ref{coupling}: The first rule of \eqref{eq:HeckeRules} says that (with $s=(z,z+1)$) if $w(z)<w(z+1),$ and the Poisson process $\mathcal{P}(z)$ has a jump, $w$ gets updated as $\sigma_{z,z+1}(w),$ whereas if $w(z)>w(z+1),$ and the Poisson process $\mathcal{P}(z)$ has a jump, $w$  stays the same with probability $1-Q,$  and gets updated as $\sigma_{z,z+1}(w)$ with probability $Q$. Note that by definition $W_{a;b}(0)=T_{id},$ so that also $w_{0}=id,$ however if we wish to start with another configuration, e.g. from a deterministic configuration $w \in S_{a;b}$, we just need to consider the permutation generated by $W_{a;b} (t) T_w$.

\subsection{Bringing into $Q-$equilibrium}\label{Qeq}
%

Let $h \in\mathcal{H}_{\mathrm{prob}}(S_{a;b})$ and $[a_1;b_1]\subseteq [a;b] $. Then \textit{bringing
 the segment $[a_1;b_1]$ into $Q-$equilibrium} means to multiply $h$ with the Mallows element $\mathcal{M}_{a_1;b_1}$ from the left. By \eqref{Melement}, this has the effect of distributing the colors of particles present in $[a_1;b_1]$ according to the stationary measure which is essentially the Mallows measure: It will formally coincide with the Mallows measure on $S_{a_1;b_1}$ if we relabel the colors of particles at positions inside the segment $[a_1;b_1]$ by integers from $[a_1;b_1]$ in a monotonous way.

Seeing $h$ as a random element of $S_{[a;b]}$, and projecting down the colors
$[a;b]$  to particles and holes, let $k'$ be the number of particles present in $[a_1;b_1].$ Then bringing into $Q-$equilibrium $[a_1;b_1]$ means that
 inside $[a_1;b_1]$, the particles and holes are distributed according to $\pi_{[a_1; b_1],k'}.$

\subsection{Idea of the proof}\label{plan}
We have recalled all the necessary notions, and will give a proof of Theorem \ref{alexey} in the following. Before this, let us briefly describe the idea of the proof.

Theorem \ref{alexey} asks us to analyze the continuous time ASEP which starts with a nontrivial initial configuration in a very precise manner. There are currently no general tools which seem to be applicable to such sorts of questions (see, however, \cite{QS20}); ASEP with step initial data being an exception where a detailed result in the form of Theorem \ref{ASEPthm} is available.  Fortunately, the algebraic framework of the Hecke algebra allows to get an \textit{exact distribution identity} which relates ASEP started with our initial configuration and the step-initial condition. It appears in the following way.

In principle, any initial configuration $w$ is generated by the element $T_w$ of the Hecke algebra. So if we are interested in the continuous time ASEP, we might want to study $W_{a;b} (t) T_w$. How to do this? The first crucial idea is that we can study 
$$
\mathfrak{i} \left( T_{w^{-1}} \mathfrak{i} \left( W_{a;b} (t) \right) \right) = \mathfrak{i} \left( T_{w^{-1}} W_{a;b} (t) \right)
$$ 
instead due to Proposition \ref{prop:CPsymmetry}. Probabilistically, it is very important (and highly non-intuitive) that we first run continuous time ASEP from the identity permutation, which can be projected to the step initial condition, and only after this we apply the multiplication by $T_{w^{-1}}$. Since we have a precise information about the step initial condition ASEP, this gives  hope to compute something about the configuration that we are interested in. Yet it is arguably impossible for general $w$. And here comes the second crucial idea --- we utilize Mallows elements which can be used to create an initial configuration which is sufficiently close to $\zeta^{0}$ (defined in \eqref{zeta01}), the particle configuration we are interested in.  The use of Mallows elements requires some estimates to relate it to the deterministic initial configuration $\zeta^{0}$ , but all of them in our situation hold with a large margin.

\subsection{Distribution identity}

Let us fix three positive integers $S,M,R$ and a positive real $t$. We will consider random walks on the Hecke algebra $\mathcal{H}_{-S-R;S+M}$. Recall that the stochastic process $W_{-S-R;S+M} (t)$ and the Mallows elements were defined in Section \ref{subsec:rand-walkHA}.

\begin{prop}
\label{prop:DistrIdent1}
With $\,{\buildrel d \over =}$ denoting equality in distribution, we have
\begin{equation*}
W_{-S-R;S+M} (t) \mathcal{M}_{-S-R;0} \mathcal{M}_{-S;S+M} \,{\buildrel d \over =}\, \mathfrak i \left( \mathcal{M}_{-S;S+M} \mathcal{M}_{-S-R;0} W_{-S-R;S+M} (t) \right).
\end{equation*}
\end{prop}
\begin{proof}
Note that   $\mathcal{M}_{-S-R;0}$, $\mathcal{M}_{-S;S+M}$, are invariant under the action of $\mathfrak{i}$.
As for $W_{a;b}(t), $ note that for arbitrary $k\geq 1$ and $s_{1},\ldots,s_k \in \mathfrak{S}_{a;b}$ we have  $\mathfrak{i}(W_{a;b}(t))=T_{s_{k}}\cdots T_{s_{1}}$ iff $W_{a;b}(t)=T_{s_{1}}\cdots T_{s_{k}},$ so that  we show $\mathfrak{i}(W_{a;b}(t))\,{\buildrel d \over =}W_{a;b}(t)$ by computing
\begin{align*}
\Pb(\mathfrak{i}(W_{a;b}(t))=T_{s_{k}}T_{s_{k-1}}\cdots T_{s_{1}})&=\Pb(W_{a;b}(t)=T_{s_{1}}T_{s_{2}}\cdots T_{s_{k}})\\&=\Pb(W_{a;b}(t)=T_{s_{k}}T_{s_{k-1}}\cdots T_{s_{1}}).
\end{align*}
Setting $a=-S-R, b=S+M$ and applying  Proposition \ref{prop:CPsymmetry}  yields the result.
\end{proof}

We will need one particular corollary of this distribution identity. Denote by $\pi_{S,R,M;t}$ the random permutation generated by $$W_{-S-R;S+M} (t) \mathcal{M}_{-S-R;0} \mathcal{M}_{-S;S+M},$$ and denote by $\hat \pi_{S,R,M;t}$ the random permutation generated by
$$\mathcal{M}_{-S;S+M} \mathcal{M}_{-S-R;0} W_{-S-R;S+M} (t).$$

\begin{cor}
\label{cor:Req}
For any $x \le y \in \R$, we have
\begin{multline}
\label{eq:DistIdent2}
\Pb \left( \min_{-S-R \le i \le 0} \pi_{S,R,M;t}^{-1} (i) > x, \max_{0 < j \le S+M} \pi_{S,R,M;t}^{-1} (j) \le y \right)
\\ = \Pb \left( \min_{-S-R \le i \le 0} \hat \pi_{S,R,M;t} (i) > x, \max_{0 < j \le S+M} \hat \pi_{S,R,M;t} (j) \le y \right).
\end{multline}
\end{cor}
\begin{proof}
Immediate from Proposition \ref{prop:DistrIdent1}.
\end{proof}

The event on the left-hand side of \eqref{eq:DistIdent2} distinguishes whether a particle is of color $\le 0$ and $>0$. Therefore, the probability of this event can be expressed as a probability of the corresponding event in a single-species ASEP. Similarly, the right-hand side of \eqref{eq:DistIdent2} distinguishes whether a particle is of color $\le x$, inside $(x;y]$, or $>y$. Therefore, the probability of this event can be expressed as a probability of the corresponding event in a two-species ASEP, i.e. ASEP which contains first-class particles, second-class particles, and holes. Let us present these simpler processes and events more formally.



\subsection{Auxiliary processes}

In this section we will introduce several auxiliary ASEP processes. All of them are continuous time processes which start from distinct initial configurations. These processes may depend on $M,R,S,x,y$, but we omit this in notations. See Figure \ref{Graph2} for an illustration.

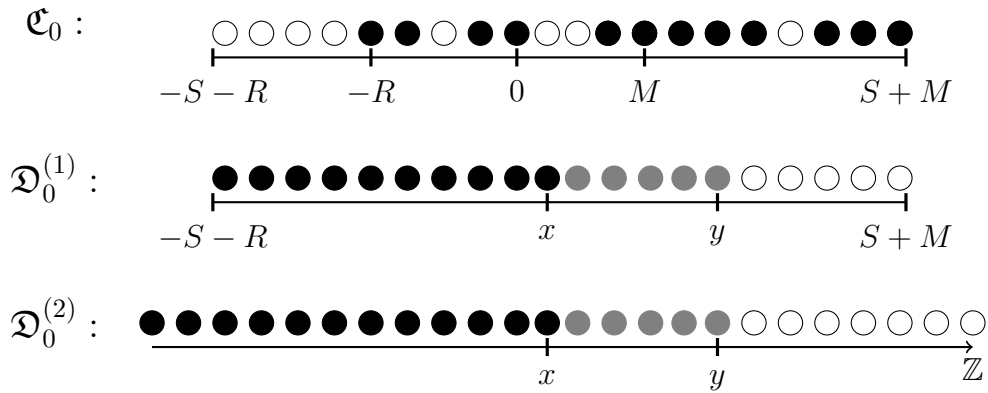
\begin{figure}\begin{center}
\begin{tikzpicture}[scale=1.6]

  \draw[thick] (-0.5,0) -- (5.2,0);

   \draw (-1.8,0.5) node[below] {\large{$\mathfrak{C}_{0}:$}};

\draw (3.05,0.2) circle (0.1);

   \draw (-0.5,-0.1) node[below] {$-S-R$};
      \draw (5.2,-0.1) node[below] {$S+M$};
 \draw (3.05,-0.1) node[below] {$M$};
  \draw (2,-0.1) node[below] {$0$};
    \draw (0.8,-0.1) node[below] {$-R$};
   \foreach \x in {3.05,2,0.8,5.2,-0.5}
      \draw[very thick] (\x,0.075)--(\x,-0.075);

  \foreach \x in {-0.1,-0.4, 4.25,1.4,2.25,2.5,0.2,0.5}
      \draw (\x,0.2) circle(0.1);
   \filldraw (1.1,0.2) circle (0.1);
   \foreach \x in {2.75,  3.35,3.65,3.95,4.55,4.85,5.15}
      \filldraw (\x,0.2) circle(0.1);
         \foreach \x in {2.75, 3.05, 3.35,3.65,3.95,4.55,4.85,5.15,1.7,2,0.8}
      \filldraw (\x,0.2) circle(0.1);

      \begin{scope}[yshift=-1.2cm]
        \draw[thick] (-0.5,0) -- (5.2,0);

   \draw (-1.8,0.5) node[below] {\large{$\mathfrak{D}_{0}^{(1)}:$}};

   \draw (-0.5,-0.1) node[below] {$-S-R$};
      \draw (5.2,-0.1) node[below] {$S+M$};
    \draw (3.65,-0.1) node[below] {$y$};
      \draw (2.25,-0.1) node[below] {$x$};
   \foreach \x in {5.2,-0.5,3.65,2.25}
      \draw[very thick] (\x,0.075)--(\x,-0.075);

         \foreach \x in {2.25,2,1.7, 1.4,1.1,0.8,0.5, -0.1,-0.4,0.2}
      \filldraw (\x,0.2) circle(0.1);
         \foreach \x in {2.5,2.8,3.1,3.375,3.65}
      \filldraw[gray] (\x,0.2) circle(0.1);
       \foreach \x in {3.95,4.25,4.55,4.85,5.15 }
      \draw (\x,0.2) circle(0.1);
 \end{scope}

    \begin{scope}[yshift=-2.4cm]
              \draw (-1.8,0.5) node[below] {\large{$\mathfrak{D}_{0}^{(2)}:$}};

  \draw[thick, ->] (-1,0) -- (5.75,0) node[below] {$\mathbb Z$};


    \draw (3.65,-0.1) node[below] {$y$};
      \draw (2.25,-0.1) node[below] {$x$};
   \foreach \x in {3.65,2.25}
      \draw[very thick] (\x,0.075)--(\x,-0.075);

         \foreach \x in {2.25,2,1.7, 1.4,1.1,0.8,0.5, -0.1,-0.4,0.2,-0.7,-1}
      \filldraw (\x,0.2) circle(0.1);
         \foreach \x in {2.5,2.8,3.1,3.375,3.65}
      \filldraw[gray] (\x,0.2) circle(0.1);
       \foreach \x in {3.95,4.25,4.55,4.85,5.15,5.45,5.75 }
      \draw (\x,0.2) circle(0.1); \end{scope}
\end{tikzpicture}\end{center}
\caption{From top to bottom: The three particle configurations $\mathfrak{C}_0, \mathfrak{D}^{(1)}_0, \mathfrak{D}^{(2)}_0. $ Black/gray/white balls represent first class particles/second class particles/holes.
 $\mathfrak{C}_0$ is random except for $p=1$ in which case
$\mathfrak{C}_0=\mathbf{1}_{[-R+1;0]}+\mathbf{1}_{[M;S+M]}$,  for $p<1,$ the particles remain within $\mathcal{O}(1)$ distance from  $[-R+1;0],[M;S+M].$
For $x<y,$ the configuration $\mathfrak{D}^{(1)}_0$ is on $[-S-R;S+M]$ and has first class particles in $[-S-R;x], $ second class particles on $(x;y], $ and holes in  $(y;S+M], $ and  $\mathfrak{D}^{(2)}_0$ is the extension of $\mathfrak{D}^{(1)}_0$ to $\Z$.
}
\label{Graph2}
\end{figure}
$\mathfrak{C}_0$ is a (random) configuration of particles and holes on $[-S-R;S+M]$ obtained in the following way. We start with particles at $[-S-R;0]$ and holes in $[1;S+M]$. First, we bring into $Q$-equilibrium the segment $[-S;S+M]$. Second, we bring into $Q$-equilibrium the segment $[-S-R;0]$.

$\mathfrak{D}^{(1)}_0$ is a (deterministic) configuration of first class particles, second-class particles and holes on $[-S-R;S+M]$ positioned in the following way: At positions $\le x$ we have first class particles, at positions inside $(x;y]$ we have second class particles, and at positions $>y$ there are holes.

$\mathfrak{D}^{(2)}_0$ is a (deterministic) configuration of first class particles, second-class particles and holes on $\Z$ positioned in essentially the same way: At positions $\le x$ we have first class particles, at positions inside $(x;y]$ we have second class particles, and at positions $>y$ there are holes.

$\mathfrak{C}_t$,  $\mathfrak{D}^{(1)}_t$, $\mathfrak{D}^{(2)}_t$ are the notations for configurations of these processes after time $t$.

$\mathfrak{\tilde D}^{(1)}_t$ is the (random) configuration of particles obtained from the configuration $\mathfrak{D}^{(1)}_t$ by  bringing into $Q$-equilibrium the segment $[-S-R;0]$ in it.  $\mathfrak{\hat D}^{(1)}_t$ is obtained from $\mathfrak{\tilde D}^{(1)}_t$ by bringing into $Q$-equilibrium the segment $[-S;S+M]$ in it. See also Figure \ref{Graph3}.

Recall that for a (possibly random) configuration $\mathfrak{A}$ we denote by $\mathcal{L} \left( \mathfrak{A} \right)$ the position of the leftmost particle in $\mathfrak{A}$, and we denote by $\mathcal{R} \left( \mathfrak{A} \right)$ the position of the rightmost hole in $\mathfrak{A}$.  We are in a position to reformulate Corollary \ref{cor:Req} in the language of these processes.

\begin{prop}
\label{prop:DistId3}
We have
\begin{multline}
\label{eq:DistIdent3}
\Pb \left( \mathcal{L} \left( \mathfrak{C}_t \right) > x, \mathcal{R} \left( \mathfrak{C}_t \right) \le y \right)
\\ = \Pb \left( \mbox{all first class particles in $\mathfrak{\hat D}^{(1)}_t$ are at positions $>0$}, \right. \\ \left. \mbox{all holes in $\mathfrak{\hat D}^{(1)}_t$ are at positions $\le 0$} \right).
\end{multline}
\end{prop}
\begin{proof}
Note that $\pi^{-1}_{S,R,M;t}$ maps types of particles into positions. If we map all colors $\le 0$ into particles, and colors $> 0$ into holes, we will obtain the coupling of $\pi^{-1}_{S,R,M;t}$ and $\mathfrak{C}_t$, which will give the expression in the left-hand side of \eqref{eq:DistIdent3}. Analogously, to obtain the right-hand side, in $\hat \pi_{S,R,M;t}$ we need to map all types $> y$ into holes, all types inside $(x;y]$ into second class particles, and all types $\le x$ into first class particles.
\end{proof}

\subsection{Limit transition }

In the remainder of the section we will consider sequences of parameters $M=M(N) = N-k_N+1$, $R=R(N) = k_N$, $x=x(N)=N- 2 k_N - N^{1/10}$, $y=y(N)=N- 2 k_N + N^{1/10} $, $t = t(N)= g(k_N,c)$. We also assume from now on  that $S=S(N)$ is an arbitrary sequence of numbers such that $S \ge N^N$. We will study the $N \to \infty$ limit. The first key result is the following.

\begin{prop}
\label{prop:baseMult}
We have
\begin{multline}
\label{prop:brr}
\liminf_{N\to\infty} \Pb \left( \mbox{all first class particles in $\mathfrak{\hat D}^{(1)}_t$ are at positions $>0$}, \right. \\ \left. \mbox{ all holes in $\mathfrak{\hat D}^{(1)}_t$ are at positions $\le 0$} \right) \ge F_{\GUE} \left( c f(\alpha) \right) .
\end{multline}

\end{prop}
\begin{proof}
Throughout the proof, $C_1 , C_2>0$ will be some constants independent of $N$ whose values are immaterial and may change from line to line.
Let $\mathcal{N}_1$ be the number of holes to the left of 0 in the configuration $\mathfrak{D}^{(1)}_t$ and let $\mathcal{N}_2$ be the number of holes to the left of 0 in the configuration $\mathfrak{D}^{(2)}_t$. Clearly, if $S$ is incomparably larger than $M$, $N$, and $t$, then these two quantities have almost the same distribution. More formally, let us couple the processes $\mathfrak{D}^{(1)}_t$ and $\mathfrak{D}^{(2)}_t$ via the basic coupling. Note that if there exists at least one position inside $[-S-R;x]$ and at least one position inside $[y;S+M]$ such that no signals from Poisson processes were received at these points, then the configurations of $\mathfrak{D}^{(1)}_t$ and $\mathfrak{D}^{(2)}_t$ inside the interval $[-S-R;S+M]$ must coincide. The probability  that a given position has not received a signal can be bounded by $C_1 \exp(-C_2 t)$, and the probability  that there is at least one such position in the interval $[-S; x]$ is bounded by $1 - (1 - C_1 \exp(-C_2 t))^{x+S}$. For  our values of parameters, this can be estimated as $1 - C_1 \exp(-C_2 N)  $. The same argument also works for the interval $[y; S+M]$. Therefore, we have $\lim_{N\to \infty}\Pb(\mathcal{N}_1 = \mathcal{N}_2)=1.$

Next, we note that
\begin{equation}\label{TWcor}
\lim_{N \to \infty} \Pb \left( \mathcal{N}_2 \ge k_N + N^{1/10} \right) = F_{\GUE} \left( c f(\alpha) \right).
\end{equation}
Indeed,  to compute   $\mathcal{N}_2$ we need not distinguish between first and second class particles. By the particle hole-duality, we then get
$$\Pb \left( \mathcal{N}_2 \ge k_N + N^{1/10} \right)=\Pb(x_{k_{N}+N^{1/10}}^{\mathrm{step}}>N-2k_{N}+N^{1/10}),
$$
so that \eqref{TWcor} follows from Corollary \ref{cor}. Since  $\lim_{N\to \infty}\Pb(\mathcal{N}_1 = \mathcal{N}_2)=1,$ we obtain
\begin{equation}
\label{eq:applyTW}
\lim_{N \to \infty}\Pb \left( \mathcal{N}_1 \ge k_N + N^{1/10} \right) = F_{\GUE} \left( c f(\alpha) \right). \end{equation}

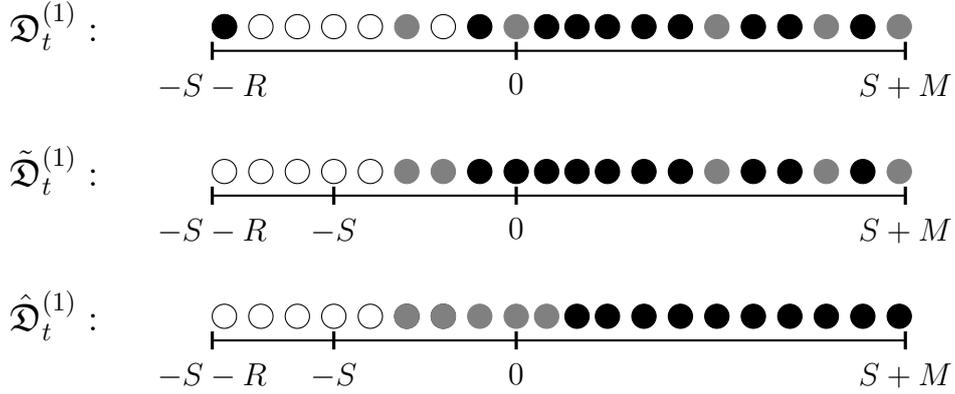
\begin{figure}\begin{center}
\begin{tikzpicture}[scale=1.6]

  \draw[thick] (-0.5,0) -- (5.2,0);

   \draw (-1.8,0.5) node[below] {\large{$\mathfrak{D}^{(1)}_t:$}};

   \draw (-0.5,-0.1) node[below] {$-S-R$};
      \draw (5.2,-0.1) node[below] {$S+M$};
  \draw (2,-0.1) node[below] {$0$};

   \foreach \x in {2,5.2,-0.5}
      \draw[very thick] (\x,0.075)--(\x,-0.075);

  \foreach \x in {-0.1,-0.4, 1.4,0.8, 0.2,0.5}
      \draw (\x,0.2) circle(0.1);
     \foreach \x in {2.75,3.05,2.5, 3.95, 2.25,3.35,4.25,4.85,1.7,-0.4}
      \filldraw (\x,0.2) circle(0.1);

          \foreach \x in {2,1.1,3.65,4.55,5.15}
      \filldraw[gray] (\x,0.2) circle(0.1);

      \begin{scope}[yshift=-1.2cm]

        \draw[thick] (-0.5,0) -- (5.2,0);

   \draw (-1.8,0.5) node[below] {\large{$\tilde{\mathfrak{D}}^{(1)}_t:$}};

 \draw (2,-0.1) node[below] {$0$};
 \draw (0.5,-0.1) node[below] {$-S$};
   \draw (-0.5,-0.1) node[below] {$-S-R$};
      \draw (5.2,-0.1) node[below] {$S+M$};

   \foreach \x in {5.2,-0.5,0.5,2}
      \draw[very thick] (\x,0.075)--(\x,-0.075);

         \foreach \x in {2.25,3.05,2,1.7,3.35,3.95,4.85,4.25,2.75,2.5 }
      \filldraw (\x,0.2) circle(0.1);

         \foreach \x in {1.1,1.4,5.15,4.55,3.65}
      \filldraw[gray] (\x,0.2) circle(0.1);
       \foreach \x in {0.8,0.5, -0.1,-0.4,0.2 }
      \draw (\x,0.2) circle(0.1);
 \end{scope}

    \begin{scope}[yshift=-2.4cm]
              \draw[thick] (-0.5,0) -- (5.2,0);

   \draw (-1.8,0.5) node[below] {\large{$\hat{\mathfrak{D}}^{(1)}_t:$}};
 \draw (2,-0.1) node[below] {$0$};
 \draw (0.5,-0.1) node[below] {$-S$};
   \draw (-0.5,-0.1) node[below] {$-S-R$};
      \draw (5.2,-0.1) node[below] {$S+M$};
    \foreach \x in {5.2,-0.5,0.5,2}
      \draw[very thick] (\x,0.075)--(\x,-0.075);
           \foreach \x in {-0.4,-0.1,0.2,0.5,0.8,1.1,1.4}
      \draw (\x,0.2) circle(0.1);
           \foreach \x in {1.4,1.1,1.7,2,2.25}
      \filldraw[gray] (\x,0.2) circle(0.1);
             \foreach \x in {2.5,2.75,3.05,3.35,3.65,3.95,4.25,4.55,4.85,5.15}
      \filldraw (\x,0.2) circle(0.1);
          \end{scope}
\end{tikzpicture}\end{center}
\caption{From top to bottom: The  particle configurations $ \mathfrak{D}^{(1)}_t,  \tilde{\mathfrak{D}}^{(1)}_t,  \hat{\mathfrak{D}}^{(1)}_t. $    If $ \mathcal{N}_1 \ge R + N^{1/10},$
then $\tilde{\mathfrak{D}}^{(1)}_t $ will have the segment $[-S-R;-S)$ filled only by holes with very high probability. Consequently, with very high probability $\hat{\mathfrak{D}}^{(1)}_t $  will have all its holes in $[-S-R;0],$ the second class  particles will be inside or close to the segment $[-N^{\frac{1}{10}}; N^{\frac{1}{10}}],$ and all first class particles will be    in $[ 0; S+R].$
}
\label{Graph3}
\end{figure}
Let us now analyze how the configuration $\mathfrak{\hat D}^{(1)}_t$ looks like conditioned on the event   $ \mathcal{N}_1 \ge k_N + N^{1/10}$. See Figure \ref{Graph3} for an illustration. According to the definition, we need to first bring into $Q$-equilibrium the segment $[-S-R;0]$ in the configuration $\mathfrak{ D}^{(1)}_t$. Since we have at least $k_N + N^{1/10} = R + N^{1/10}  $ holes in the segment $[-S-R;0]$, Proposition \ref{3.2} allows to conclude that the positions in the interval $[-S-R;-S)$ will \textit{all} be filled by holes with  probability at least $1 - C_1\exp \left( - C_2 N^{1/10} \right)$ after this step. Let us further restrict ourselves on the event that indeed $[-S-R;-S)$ is filled by holes. Then in the remaining segment $[-S;S+M]$ we have exactly
$$
S+M -  y  -R = S +(N-k_N+1) - (N-2 k_N +  N^{1/10}  ) - k_N = S -  N^{1/10}+1
$$
holes. Our second (and last) step is to bring into $Q$-equilibrium the segment $[-S;S+M]$. Again, due to Proposition \ref{3.2} the conditional probability that all these $S -  N^{1/10}+1 $ holes are inside  $[-S;0]$ can be bounded from below as  $1 - C_1 \exp \left( - C_2 N^{1/10} \right)$. Also, the segment $[-S;S+M]$ contains all $S+N-k_N-  N^{1/10} +1$ first class particles of the system, and the conditional probability that all of them will be inside the segment $[0;S+M]$  is also at least  $1 - C_1 \exp \left( - C_2 N^{1/10} \right)$ again by Proposition \ref{3.2}. It remains to note that the conditional probability that all the mentioned three events happen can also be estimated as $1 - C_1 \exp \left( - C_2 N^{1/10} \right)$, and that if these three events happen, than the event in the lefthand side of \eqref{prop:brr} also happens: The holes are all to the left of 0, and all first-class particles are to the right of 0. Combining this with \eqref{eq:applyTW} implies the claim.
\end{proof}

We have now collected all necessary ingredients to prove Theorem \ref{alexey}.
\begin{proof}[Proof of Theorem \ref{alexey}]
We define the configuration
$$
\widehat{\zeta}^0:= \mathbf{1}_{[ -k_N+1; 0]  } + \mathbf{1}_{\Z_{>(N-k_N)}}.
$$
Note that $\widehat{\zeta}^0$ is simply a shift by $(-k_N)$ of the configuration $\zeta^0$ from \eqref{zeta01}. We denote further the restriction of $\widehat{\zeta}^0$ to $[-S-R;S+M]$ by $\widehat{\zeta}^{0,S},$ i.e.
$$
\{0,1\}^{[-S-R;S+M]}\ni \widehat{\zeta}^{0,S}:= \mathbf{1}_{ [-k_N+1;  0] } + \mathbf{1}_{[N-k_{N}+1; S+ N-k_{N}+1]}.
$$
Let us prove that
\begin{equation}\label{60}
\lim_{N\to \infty}\Pb(\mathcal{L}(\widehat{\zeta}^{0,S}_{t})=\mathcal{L}(\widehat{\zeta}^{0}_{t}), \mathcal{R}(\widehat{\zeta}^{0,S}_{t})=\mathcal{R}(\widehat{\zeta}^{0}_{t}))=1.
\end{equation}
To see this,  note that for the event $\{\mathcal{L}(\widehat{\zeta}^{0,S}_{t})=\mathcal{L}(\widehat{\zeta}^{0}_{t}), \mathcal{R}(\widehat{\zeta}^{0,S}_{t})=\mathcal{R}(\widehat{\zeta}^{0}_{t})\}$ to happen, it suffices that one of the Poisson processes in $[-S-R;-S/2]$ and one of the Poisson processes in $[S/2;S+M]$ make no jump during $[0,t].$ Since $S\geq N^{N},$ this will happen with probability going to $1$ as $N\to \infty$  and \eqref{60} follows.

Recall further the partial order \eqref{order2} and note that we have
\begin{equation} \label{smaller}
\mathfrak{C}_{0}\preceq \widehat{\zeta}^{0,S}.
\end{equation}
Now, one can easily  compute that
under the basic coupling,  \eqref{smaller} implies that
\begin{equation}\label{shorter}
 \mathcal{L}(\mathfrak{C}_{t})\leq \mathcal{L}(\widehat{\zeta}^{0,S}_{t}) \mathrm{\,\,and\,\, } \mathcal{R}(\mathfrak{C}_{t})\geq \mathcal{R}(\widehat{\zeta}^{0,S}_{t})
\end{equation}
 (see e.g. \cite[Lemma  3.2]{N20AAP} and  \cite[Lemma  3.2]{N20CMP} for a proof for the partial order \eqref{order} which carries over to our case).
 We can now conclude
  \begin{align*}
 \liminf_{N\to\infty} \Pb(B_{N}(c))&= \liminf_{N\to\infty}\Pb(\mathcal{L}(\widehat{\zeta}^{0,S}_{t})> x, \mathcal{R}(\widehat{\zeta}^{0,S}_{t})\leq y )
 \\&\geq \liminf_{N\to\infty} \Pb \left( \mathcal{L} \left( \mathfrak{C}_t \right) >  x, \mathcal{R} \left( \mathfrak{C}_t \right) \le y \right)
 \\& \geq F_{\GUE} \left( c f(\alpha) \right),
\end{align*}
where the first inequality follows from \eqref{shorter}, the second follows from Propositions \ref{prop:DistId3} and \ref{prop:baseMult}.

Finally, using Corollary \ref{cor} we obtain
 \begin{align*}
 \limsup_{N\to\infty} \Pb(B_{N}(c))&\leq  \lim_{N\to\infty}\Pb(\mathcal{L}(\widehat{\zeta}^{0}_{t})\geq x)\\&\leq \lim_{N \to \infty}\Pb\left(x_{k_{N}}^{\mathrm{step}}(g(k_{N},c))\geq N-2k_{N} -N^{1/10}\right)\\&=F_{\GUE} \left( c f(\alpha) \right),
 \end{align*}
 so that Theorem \ref{alexey} follows.

\end{proof}

\begin{remark} \label{rem:further-questions}
It is plausible that an extension of Theorem \ref{main} to the case of slowly varying parameters $p-q \to 0$, $N^{1/3} (p-q) \to \infty$ can be obtained by an upgrade of our method. This will require certain additional estimates, in particular, regarding the asymptotics of Mallows measure. On the other hand, if $(p-q)$ becomes to decay to 0 faster, new effects should appear, first on the level of the cutoff window, and then on the level of the cutoff itself (cf. \cite{LP16}). 

Another interesting direction would be to understand whether our method can be adjusted to the study of mixing times in ASEP with open boundaries, cf. \cite{GNS20}.     
\end{remark}

\bibliography{Biblio}{}
\bibliographystyle{plain}

\end{document}